\documentclass[11pt,a4paper]{scrartcl}
\usepackage{amsmath}
\usepackage{amsthm}
\usepackage{amsfonts}
\usepackage{amssymb}
\usepackage{graphicx}

\usepackage[T1]{fontenc}

\newtheorem{thm}{Theorem}

\newtheorem{lem}{Lemma}
\newtheorem*{rem}{Remark}
\newtheorem{example}{Example}

\def\spann  {\mathrm{span}}
\def\ohne   {\setminus}
\def\nmin   {n_{\min}}
\def\div    {\textnormal{div}}
\def\R      {\mathbb{R}}
\def\H      {\ensuremath{\op{H}}}
\def\C      {\mathbb{C}}
\def\N      {\mathbb{N}}
\def\ud     {\,\textnormal{d}}
\def\eps    {\varepsilon}
\def\diam   {\textnormal{diam}}
\def\supp   {\textnormal{supp}\,}
\newcommand{\op}[1]{\mathcal{#1}}
\def\fie    {\varphi}
\def\rank   {\textnormal{rank}\,}

\newcommand{\mcO}{\ensuremath{\mathcal{O}}}

\newcommand{\norm}[1]{\ensuremath{\|#1\|}}
\newcommand{\dist}{\textnormal{dist}}

\renewcommand{\L}[1]{\op{L}(#1)}

\newcommand{\e}{{\textnormal{e}}}

\newcommand{\SPND}[2]{({#1},{#2})_{L^2(D)}}

\newcommand{\tnrm}[1]{|\!|\!|#1|\!|\!|}

\begin{document}
\title{Low-rank approximation of elliptic boundary value problems with
  high-contrast coefficients\footnote{This work was supported by the
    DFG collaborative research center SFB~1060.}}
\author{M.~Bebendorf\footnote{Institut f\"ur Numerische Simulation,
Universit\"at Bonn, Wegelerstrasse 6, 53115 Bonn, Germany}}
\date{\today}
\maketitle

\begin{abstract}
We analyze the convergence of degenerate approximations to Green's function of elliptic boundary value problems with high-contrast
coefficients. It is shown that the convergence is independent of the contrast if the error is measured with respect to suitable norms.
This lays ground to fast methods (so-called hierarchical matrix approximations) which do not have to be adapted to the coefficients.
\end{abstract}


\section{Introduction}
Elliptic problems with non-smooth, high-contrast coefficients appear in many
fields of science ranging from the simulation of porous media and
composite materials to the recent field of uncertainty quantification.
The numerical solution of such problems is challenging if all details
of the physical problems are to be resolved due to the  large number 
of degrees of freedom needed for a sufficiently accurate
discretization. The enormous amount of computer memory
and CPU time can be reduced to some extend if one is
satisfied with macroscopic properties of the
solution. The \textit{multiscale finite element method}~\cite{HW97,EH09} and
the \textit{heterogeneous multiscale method}~\cite{EE03} capture
small-scale effects on large scales. These methods rely on special
assumptions on the coefficient such as self-similarities, periodicity
and scale separation.
If such properties cannot be exploited, then the discretization has to
be done with full detail. In this case, the numerical method used to
solve the problem has to be efficient and robust with respect to the
operator's coefficients.

Methods that achieve a computational complexity that scales linearly
with the number of degrees of freedom often rely on multiscale
techniques, too. The \textit{multi-grid method\,}~\cite{MG,BrambleMG,WiNeJa01} relaxes the
error at different scales on coarser grids. \textit{Algebraic multigrid
methods}~\cite{AMG} try to achieve the robustness with respect to
non-smooth coefficients by mostly heuristic strategies.
Another successful class of methods, which are well-suited also for
parallelization, are domain decomposition methods such as the \textit{finite element tearing and interconnect method}~\cite{FETI,FLLPR01}. Although significant progress has been made in making the method robust with respect to variable coefficients (see \cite{MR1367653,klWiDr96,B7:PechSch08,PeSch09a,PeSaSch13}),
the theory still contains assumptions on the coefficient's distribution. 

The \textit{mosaic skeleton method}~\cite{MR99f:15005} and
\textit{hierarchical ($\mathcal{H}$-) matrices}~\cite{MR2000c:65039,MR2001i:65053} are
historically related with the \textit{fast multipole
  method}~\cite{MR88k:82007,MR99c:65012}, which
can also be regarded as a multiscale method. Since $\H$-matrices are based
on substructuring and low-rank approximation rather than multiscale effects, they are
closer related with the \textit{tree code
  algorithm}~\cite{BarnesHut}. While fast multipole methods and the
tree code algorithm are designed only for efficiently applying a non-local
operator to a vector, $\H$-matrices provide an approximate algebra in
which
approximations of fully populated matrices (such as integral operators
and the inverse or the
factors of the LU decomposition of sparse matrices) can be computed with
logarithmic-linear complexity. Although one could use $\H$-matrix
approximations as direct solvers, it is usually more
efficient to employ them as approximate preconditioners.

The existence of $\H$-matrix approximations to the
inverse of finite element~(FE) discretizations was proved in \cite{MBWHInv}
for boundary value problems of scalar elliptic operators
\begin{equation}\label{eq:op1}
\op{L} :=-\div\, A(x)\nabla
\end{equation}
with a symmetric positive-definite coefficient~$A(x)\in\R^{d\times d}$ with $a_{ij}\in L^\infty(\Omega)$.
The proof is based on the existence of
exponentially convergent degenerate approximations 
\begin{equation} \label{eq:degap}
G_k(x,y):=\sum_{i=1}^k u_i(x) v_i(y)\approx G(x,y)
\end{equation}
with suitable functions $u_i$, $v_i$ to Green's function~$G$ for $\op{L}$
and the computational domain $\Omega\subset\R^d$.
Due to the exponential convergence of the approximation~\eqref{eq:degap},
the accuracy~$\eps$ enters the length~$k$ of the sum in~\eqref{eq:degap} only
logarithmically. Furthermore, the proof reveals that $k$ does not depend on the smoothness of~$A$.
However, the bound on $k$ still depends significantly on the contrast
\[
\kappa_A:=\frac{\lambda_{\max}}{\lambda_{\min}},
\]
i.e.\ the ratio of the coefficient's largest
and
smallest eigenvalue
$\lambda_{\max}:=\text{ess sup}_{x\in\Omega} \,\lambda_{\max}(x)$ and
$\lambda_{\min}:=\text{ess inf}_{x\in\Omega} \,\lambda_{\max}(x)$, where
\[\lambda_{\max}(x):=\max_{\lambda\in\sigma(A(x))}\lambda,\quad
\lambda_{\min}(x):=\min_{\lambda\in\sigma(A(x))}\lambda,
\]
and $\sigma(A)$ denotes the spectrum of the matrix~$A$.
The size~$k$ determines the overall complexity of the $\H$-matrix approximation.
Despite the dependence of $k$ on $\kappa_A$ in theory, an impact of high-contrast
coefficients~$A$ on the efficiency of $\H$-matrices has never been
observed in practise.
The aim of this article is to underpin this observation
theoretically. To this end, a norm will be introduced that depends on
the coefficient~$A$ and generalizes the \textit{flux norm}~\cite{BeOw10}. For the construction of
approximations~\eqref{eq:degap}
it will be shown with respect to this norm that  the Kolmogorov $k$-width
of the space of $\op{L}$-harmonic
functions decays exponentially with~$k$ and does not depend on
the contrast, i.e., it will be shown that
\[
k\sim|\log\eps|^{d+1},
\]
whereas in \cite{MBWHInv} we proved
$k\sim\kappa_A^{d/2}|\log\eps|^{d+1}$ with respect to the $L^2$-norm.
While the focus of this article is on the dependence of $k$ on the
contrast in diffusion problems, the recent publication \cite{EnZh04} analyzes $k$ for
high-frequency Helmholtz problems.

The approximation~\eqref{eq:degap} of Green's function can be
used to prove the existence of low-rank approximations to the inverse of
FE discretizations. 
Since the inverse of $\op{L}$ has the representation
\[
(\op{L}^{-1}\fie)(x)=\int_\Omega G(x,y)\fie(y)\ud y,\quad x\in\Omega,
\]
the existence of a degenerate approximation~\eqref{eq:degap}
on a pair of domains $D_1\times D_2\subset\Omega\times\Omega$ leads to the existence of rank-$k$ approximations
\[
(UV^T)_{ij}\approx (\op{L}^{-1}\fie_j,\fie_i)_{L^2(\Omega)},\quad
U_{i\ell}:=\int_\Omega u_\ell\fie_i\ud x,\;V_{j\ell}:=\int_\Omega v_\ell\fie_j\ud x,
\]
to the FE~discretization of $\op{L}^{-1}$ provided that the support of the FE
basis functions~$\fie_i$ and $\fie_j$ satisfy $\supp\fie_i\subset D_1$ and $\supp\fie_j\subset D_2$.
For the matrix approximation result, however, usual
$L^2$-norm  estimates are required as the matrix error is measured with respect
to the spectral or the Frobenius norm. Changing the norm in the final
estimate to the $L^2$-norm introduces a contrast-dependent term in the
error estimate, which due to the exponential convergence enters the
matrix rank~$k$ only logarithmically.

Note that the approximation technique
presented in this article applies not only to operators \eqref{eq:op1}.
In \cite{InvFEM2} we considered general second order elliptic scalar
operators and in \cite{bebeMaxwell} an analogous result was proved for the
curl-curl operator.
In practice, the LU factorization can be used to solve linear systems
significantly faster than the inverse of a matrix. It is therefore important  
to remark that also the Schur complement and the factors of the LU decomposition can be approximated by
hierarchical matrices with logarithmic-linear complexity; see \cite{MBLU} for a
proof. 
The proof in \cite{MBLU} is based on the approximation of the inverse. Hence,
the results on the approximation of the factors of the LU
decomposition directly benefit from the new estimates of this
article. Notice that hierarchical matrix approximations, from an
algorithmic point of view, are constructed independently of the
operator. Hence, this class of methods provides a fast and robust
approach to problems with non-smooth, high-contrast coefficients.

The article is organized as follows.
The way matrices are subdivided into sub-blocks is crucial for the efficiency
of hierarchical matrices. Furthermore, the block structure is responsible for
the properties of the domains on which Green's function is to be
approximated. In Sect.~\ref{sec:hmat}, we will therefore shortly review
the structure of hierarchical matrices. Sect.~\ref{sec:wpoin} contains the
new low-dimensional approximation result with contrast-independent
constant. To prove it, a suitable norm will be introduced in Sect.~\ref{sec:flux}. 
The existence of degenerate approximations to Green's function
on pairs of domains will be treated in Sect.~\ref{sec:innreg}.
This will be based on interior regularity estimates and on the
transfer property of flux norms.
Numerical experiments in Sect.~\ref{sec:num} support our theoretical
findings.

\section{Hierarchical matrices} \label{sec:hmat}
The setting in which approximations of solution operators will be
approximated in this article are
hierarchical matrices. This methodology introduced by
Hackbusch et al.\ \cite{MR2000c:65039,MR2001i:65053} is designed to
handle fully populated matrices such as approximations to the inverse or the
factors of the LU decomposition with logarithmic-linear complexity; see
\cite{Bebendorf:2008,wh09}.

The efficiency of hierarchical matrices is based on low-rank approximations
of each sub-matrix of a suitable partition $P$ of the full set of matrix
indices $I\times I$, $I:=\{1,\dots,n\}$.
The construction of $P$ has to account for two aims. On one hand, it has
to guarantee that the rank $k$ of the approximation
\begin{equation} \label{eq:Plwr}
B_{ts} \approx XY^T, \quad X\in \C^{|t|\times k},\, Y\in \C^{|s|\times k},
\end{equation}
to each block $B_{ts}$, $t\times s\in P$, of a given matrix $B\in\C^{I\times I}$
depends logarithmically on its
approximation accuracy. Here, $B$ denotes a fully populated matrix, e.g., the
inverse of a stiffness matrix $A$ resulting from
FE discretization. On the other hand, $P$
must be computable with logarithmic-linear complexity.
The former issue will be addressed by the so-called \textit{admissibility
condition} in Sect.~\ref{sec:blcltree}, while the latter problem can be
solved by so-called \textit{cluster trees}.

\subsection{Cluster tree}
Searching the set of possible partitions of $I\times I$ for a partition $P$
which guarantees \eqref{eq:Plwr}
seems practically impossible since this set is considerably large.
By restricting ourselves to blocks $t\times s$ made up from rows
$t$ and columns $s$ which are generated by recursive subdivision, $P$ can be
found with almost linear complexity.
The structure which describes the way $I$ is subdivided into smaller parts
is the cluster tree. 
A tree $T_I$ is called a \emph{cluster tree} for an index set $I$ if it
satisfies the following conditions:
\begin{enumerate}
\renewcommand{\labelenumi}{(\roman{enumi})}
\item $I$ is the root of $T_I$;
\item if $t\in T_I$ is not a leaf, then $t$ is a disjoint union of
its sons $S_I(t)=\{t_1,t_2\}\subset T_I$.
\end{enumerate}
We denote the set of leaves of the tree $T_I$ by $\L{T_I}$.

A cluster tree for $I$ can be computed, for instance, by the \textit{bounding box
method} or the \textit{principal component analysis}
\cite{Bebendorf:2008}. The latter methods take into account the geometric information
associated with the matrix indices. A nested dissection approach
\cite{bebeAlgInv} based on the
matrix graph often leads to significantly
better results.

\subsection{Block cluster tree} \label{sec:blcltree}
The approximation results from \cite{MBWHInv,InvFEM2} show that in order
to be able to guarantee a sufficient approximation of each sub-matrix
$B_{ts}$, $t\times s\in P$, of $B$ by a matrix of low rank, the
sub-block $t\times s$ has to satisfy the so-called \emph{admissibility
condition}
\begin{equation}  \label{eq:admcond}
\min\{\diam\,X_t, \diam\,X_s\}\leq \eta\,\dist(X_t,X_s)
\end{equation}
for a given parameter $\eta>0$ or $\min\{|t|,|s|\}\leq \nmin$ holds for
a given block size parameter $\nmin\in\N$. Here,
$$\diam\,X_t:=\sup_{x,y\in X_t} |x-y|\quad \text{and}\quad
\dist(X_t,X_s):=\inf_{x\in X_t,\,y\in X_s} |x-y|$$
and the support $X_t:=\bigcup_{i\in t} X_i$ of a cluster $t\in T_I$ is the
union of the supports $X_i:=\supp\fie_i$ of the basis functions $\fie_i$,
$i\in t$, corresponding to its indices.
Notice that in order to satisfy
\eqref{eq:admcond}, the distance of the supports of $t$ and $s$ has to be large
enough. This condition is caused by the fact that Green's functions
of elliptic differential operators are singular for $x=y$ only.

The partition is usually generated by recursive subdivision of $I\times I$
descending the \textit{block cluster tree} $T_{I\times I}$,
which is a cluster tree for the set of matrix indices $I\times I$ associated
with the descendant mapping $S_{I\times I}$ defined by
\[
S_{I\times I}(t,s) :=
\begin{cases}
\emptyset, & \text{if } S_I(t)=\emptyset \text{ or } S_I(s)=\emptyset, \\
S_I(t)\times S_I(s), & \text{else.}
\end{cases}
\]
The recursion stops in blocks which satisfy \eqref{eq:admcond} 
or which are small enough. The set of leaves
$\L{T_{I\times I}}$ of the block cluster tree $T_{I\times I}$
forms a partition $P$ of $I\times I$.

With a partition $P$ constructed as above, the set of $\H$-matrices with blockwise rank $k$ is defined as
$$\H(P,k):=
\{M\in\R^{I\times I}:\rank M_b\leq k\text{ for all }b\in P\}.$$
The storage requirement for $B\in \H(P,k)$ is of the order $kn\log{n}$.
Multiplying $B$ by a vector can be done with $\mcO(kn\log{n})$ arithmetic
operations. Since the sum of two $\H$-matrices $B_1,B_2\in\H(P,k)$ exceeds
blockwise rank $k$, the sum has to be truncated to $\H(P,k)$. This
can be done with complexity $\mcO(k^2n\log{n})$ if an approximation
error of controllable size can be tolerated.
The complexity of computing an approximation to the product of two
$\H$-matrices is $\mcO(k^2n(\log{n})^2)$;
see~\cite{MR2000c:65039,MR2001i:65053,GRHA03}.

\section{Weighted norms and interior regularity} \label{sec:wpoin}
The existence of finite-dimensional approximation spaces to the
following space will turn out to be crucial for the existence of low-rank
approximations of the discrete inverse.

Let $D\subset\Omega$ be a domain. We investigate the approximation of functions from the space
\[
X(D)=\{u\in H^1(D): a(u,\fie)=0\text{ for all }\fie\in H_0^1(D),\, u=0\text{ on }\partial D\cap\partial\Omega\}
\]
of $\op{L}$-harmonic functions vanishing on $\partial D\cap \partial \Omega$, 
where $a:H^1(D)\times H^1(D)\to\R$,
\[
a(v,w):=\int_D \nabla w^T A\nabla v\ud x,
\]
denotes the bilinear form associated with $\op{L}$.
The aim of this section is to construct a $k$-dimensional approximation space $X_k\subset L^2(K)$ which provides
$A$-independent error estimates, i.e., the Kolmogorov $k$-width
of $X(D)$ is bounded by
\begin{equation}\label{eq:apprxu}
\sup_{u\in X(D)} \inf_{v\in X_k} \frac{\norm{u-v}_K}{\norm{u}_D}\leq \eps_k
\end{equation}
with some norm $\norm{\cdot}_K$ on~$K\subset D$
and $\eps_k\leq cq^{k/d}$, where $c>0$ and $0<q<1$ are independent of~$A$.
The construction of the space~$X_k$ will later be the basis for the
construction of degenerate approximations~\eqref{eq:degap}.

In \cite{MBWHInv} we employed the Poincar\'e inequality on each piece $K_i$, $i=1,\dots,k$, of a sufficiently fine partition of~$K$, i.e.\ $\norm{u-\mu_i}_{L^2(K_i)}\leq c_i\,\diam\,K_i\norm{\nabla u}_{L^2(K_i)}$ with some $\mu_i\in\R$, $i=1,\dots,k$,
which leads to
\[
\norm{u-\tilde{u}}_{L^2(K)}\leq \eps_k\norm{\nabla u}_{L^2(K)},\quad \eps_k:=\max_{i=1,\dots,k} c_i \,\diam\,K_i,
\]
with $\tilde{u}\in X_k$ defined as $\tilde{u}|_{K_i}:=\mu_i$ and $X_k$ the space of piecewise constant functions.
The Caccioppoli-type inequality (see the remark after Lemma~\ref{lem:ipoi})
\begin{equation}\label{eq:cacc2}
\norm{\nabla u}_{L^2(K)}\leq \frac{c_T}{\sigma(K,D)}\norm{u}_{L^2(D)}
\end{equation}
leads to the desired estimate~\eqref{eq:apprxu} provided that
the distance
\[\sigma(K,D):=\dist(K,\partial D\cap\Omega)\]
of $K\subset D$ to $\partial D$ within~$\Omega$ is positive.

The error estimate, however, is not $A$-independent as
the $A$-dependence of $c_T$ in \eqref{eq:cacc2} seems to be unavoidable.
A remedy for the latter difficulty is to use weighted norms. 
\begin{lem} \label{lem:ipoi}
Let $D\subset\Omega$.
Then for any set $K\subset D$ satisfying
$\sigma(K,D)>0$ it holds that
\[\norm{A^{1/2}\nabla u}_{L^2(K)}
\leq \frac{2}{\sigma(K,D)}\norm{\lambda^{1/2}_{\max} u}_{L^2(D)}\]
for all $u\in X(D)$.
\end{lem}
\begin{proof}
Let $\xi\in C^1(D)$ such that $0\leq\xi\leq 1$, $\xi=1$ in $K$,
$\xi=0$ in a neighborhood of $\partial D\cap\Omega$ and
$$\norm{\nabla\xi}_{\infty,D}\leq \frac{2}{\sigma(K,D)}.$$
We have $\xi^2 u\in H^1(D)$ and
$\xi^2 u=0$ on $\partial D=(\partial D\cap\Omega)\cup
(\partial D\cap\partial\Omega)$. Hence, $\xi^2 u$ can be used as a test
function in the definition of $X(D)$, which leads to $a(u,\xi^2u)=0$.
Using $\nabla(\xi v)=\xi\nabla v+v\nabla\xi$ for
$v\in H^1(D)$, from
\begin{align*}
\SPND{A\nabla u}{\nabla(\xi^2 u)}
&= \SPND{A\nabla u}{\xi\nabla(\xi u)+\xi u\nabla\xi}
= \SPND{\xi A\nabla u}{\nabla(\xi u)+u\nabla\xi}\\
&= \SPND{A\nabla(\xi u)-uA\nabla\xi}{\nabla(\xi u)+u\nabla\xi}\\
&= \SPND{A^{1/2}\nabla(\xi u)-uA^{1/2}\nabla\xi}{A^{1/2}\nabla(\xi u)+uA^{1/2}\nabla\xi}\\
&= \norm{A^{1/2}\nabla(\xi u)}^2_{L^2(D)}-\norm{uA^{1/2}\nabla\xi}_{L^2(D)}^2
\end{align*}
we obtain that
\begin{align*}
\norm{A^{1/2}\nabla u}_{L^2(K)}^2\leq \norm{A^{1/2}\nabla(\xi u)}_{L^2(D)}^2
=\norm{u A^{1/2}\nabla\xi}_{L^2(D)}^2
\leq \norm{\lambda^{1/2}_{\max} u}^2_{L^2(D)}\norm{\nabla\xi}_{\infty,D}^2.
\end{align*}
\end{proof}

\begin{rem}
Using
$\lambda_{\min}^{1/2}\norm{\nabla u}_{L^2(K)}\leq \norm{A^{1/2}\nabla u}_{L^2(K)}$,
Lemma~\ref{lem:ipoi} implies that
\begin{equation}\label{eq:uc}
\norm{\nabla u}_{L^2(K)}
\leq \frac{2\sqrt{\kappa_A}}{\sigma(K,D)}\norm{u}_{L^2(D)}
\end{equation}
for all $u\in X(D)$.
\end{rem}

Assume
for a moment that~$A(x)=\alpha(x) I$ with some positive~$\alpha\in
L^\infty(\Omega)$
and define
\[
\norm{u}_{L_\alpha^2(\Omega)}:=\left(\int_\Omega \alpha |u|^2\ud x\right)^{1/2}.
\]
Due to Lemma~\ref{lem:ipoi}, we obtain the following Caccioppoli inequality with $\alpha$-independent constant 
\[
\norm{\nabla u}_{L^2_\alpha(K)}
\leq \frac{2}{\sigma(K,D)}\norm{u}_{L^2_\alpha(D)},\quad u\in X(D).
\]
Although the weight-independence of the Caccioppoli inequality can be
achieved with respect to the $\norm{\cdot}_{L^2_\alpha}$-norm,
the independence of the Poincar\'e constant
\[
\sup_{u\in H^1(K)} \inf_{\mu\in\R} \frac{\norm{u-\mu}_{L^2_\alpha(K)}}{\norm{\nabla u}_{L^2_\alpha(K)}}
\]
of the weight $\alpha$ is non-trivial.
Weighted Poincar\'e inequalities with weight-independent constant have been
presented in \cite{EfGa09,PeSch09}
for the case that the domain~$K$ is partitioned into a
finite number of disjoint Lipschitz domains $K_i\subset K$,
$i=1,\dots,r$, and $K_0:=K\setminus \bigcup_{i=1}^r \overline{K}_i$
in each of which the weight $\alpha$ is constant, i.e.
\[
\alpha(x)=\alpha_i,\quad x\in K_i,
\]
with given numbers $\alpha_i>0$, $i=0,\dots,r$.
In \cite{EfGa09}, the case that $K_i$ are inclusions of
the domain~$K_0$ satisfying $\alpha_i|K_i|\leq
\alpha_0|K_0|$, $i=1,\dots,r$, is treated. It is proved that the
Poincar\'e constant is weight-independent, but it depends on
the number~$r$ of domains. 
In \cite{PeSch09} a monotonicity of the sequence of coefficients
is assumed in the spirit of \cite{MR1367653,klWiDr96}, for which a
weight-independent Poincar\'e constant can be proved.
The independence of $\alpha$ does not hold in general. This can
be seen from the following example.
\begin{example}
Let $K=(-3,3)^2\subset\R^2$, $K_1=(-2,-1)\times (-1,1)$, and
$K_2=(1,2)\times (-1,1)$. For
\[
\alpha(x)=\begin{cases}
\delta^{-1}, & x\in K_1\cup K_2,\\
1, & \text{else},
\end{cases}
\quad\text{and}\quad
u(x)=\begin{cases}
x_1, & |x_1|\leq 1,\\
1, & 1<x_1<3,\\
-1, & -3<x_1<-1,
\end{cases}
\]
with small $\delta>0$ we obtain for arbitrary $\mu\in\R$
\[
\int_K \alpha |u-\mu|^2\ud x
\geq \int_{K_1} \delta^{-1}(-1-\mu)^2\ud x
+\int_{K_2} \delta^{-1}(1-\mu)^2\ud x=4\delta^{-1}(1+\mu^2)\geq 4\delta^{-1}
\]
and
\[
\int_K \alpha |\nabla u|^2\ud x=12,
\]
which shows
\[
\sup_{u\in H^1(K)} \inf_{\mu\in\R}
\frac{\norm{u-\mu}_{L^2_\alpha(K)}}{\norm{\nabla u}_{L^2_\alpha(K)}}
\geq \frac{1}{\sqrt{3\delta}}.
\]
Therefore, weight-dependent Poincar\'e constants for general
coefficients are unavoidable.
\end{example}

Our aim is to derive approximation estimates \eqref{eq:apprxu} which
do not depend on the coefficient~$A$. Since this cannot be
achieved by a weighted Poincar\'e inequality for general coefficients, in the rest of this chapter a more sophisticated approach to the construction
of a finite-dimensional approximation space~$X_k$ will be presented.

\subsection{Flux norms} \label{sec:flux}
In this section we introduce a norm that will be useful for proving 
estimates of the form~\eqref{eq:apprxu}  with $A$-independent
constants. Note that we consider arbitrary symmetric positive-definite
coefficient matrices~$A(x)\in\R^{d\times d}$. For $v\in H^1(D)$ and
any domain $K\subset D$ with non-empty interior define the $A$-dependent function
\[
\tnrm{v}_{D,K}:=\norm{\phi^D_v}_{1,K},
\]
where \[\norm{\phi}_{1,K}^2:=\norm{\phi}_{L^2(K)}^2+(\diam\,K)^2\,\norm{\nabla\phi}^2_{L^2(K)}\]
and $\phi^D_v\in H^1(D)$ denotes the solution of the Dirichlet problem
\begin{alignat*}{3}
-\Delta\phi_v^D&=\op{L} v &\quad&\text{in }D,\\
\phi_v^D&=v&\quad&\text{on }\partial D,
\end{alignat*}
i.e.\ we have $\phi_v^D=v+z_v^D$, where $z_v^D\in H_0^1(D)$ satisfies
\begin{equation}\label{eq:transfp}
a_\Delta(z_v^D, \fie)=a(v,\fie)-a_\Delta(v,\fie)\quad \text{for all }\fie\in H_0^1(D),
\end{equation}
where $a_\Delta(u,v):=\int_\Omega\nabla u\nabla v\ud x$ denotes the
bilinear form associated with the Laplacian.
The existence and uniqueness of $\phi_v^D\in H^1(D)$ follows from the
Lax-Milgram theorem.
In addition to $\tnrm{\cdot}_{D,K}$, we introduce the
$A$-dependent bilinear form
$(\cdot,\cdot)_{D,K}:H^1(D)\times H^1(D)\to\R$ as
\[
(v,w)_{D,K}:=(\phi_v^D,\phi^D_w)_{L^2(K)}+(\diam\, K)^2\, (\nabla\phi_v^D,\nabla\phi^D_w)_{L^2(K)},
\]
which induces $\tnrm{\cdot}_{D,K}$.

\begin{rem}
The function~$\tnrm{\cdot}_{D,K}$ is a generalization of the so-called
\emph{flux norm} introduced in~\cite{BeOw10}. The latter is defined
as~$\norm{\nabla \phi_v^D}_{L^2(D)}$ for $v\in H^1_0(D)$.
In the following, $D$ will be fixed, while the dependence of~$\tnrm{\cdot}_{D,K}$ on $K$ is of
particular importance for our needs. In particular, we cannot restrict
ourselves to~$H^1_0(D)$. 
\end{rem}

Notice that $\tnrm{v}_{D,K}$ evaluates $\phi_v^D$ only
on~$K\subset D$. Hence, $\tnrm{\cdot}_{D,K}$ cannot be a norm on
$H^1(D)$. The following lemma states that $\tnrm{\cdot}_{D,K}$ is a
semi-norm on~$H^1(D)$ and a norm on $X(D)\subset H^1(D)$.
This is due to the fact that $v\in
X(D)$ implies that $\phi_v^D$ is harmonic and Lemma~\ref{lem:nrmKD}
(see Appendix) implies
\begin{equation}\label{eq:ED2}
\tnrm{v}_{D,D}=\norm{\phi_v^D}_{1,D}\leq c_{D,K}\norm{\phi_v^D}_{1,K}=c_{D,K}\tnrm{v}_{D,K}.
\end{equation}

\begin{lem} \label{lem:skp}
The bilinear from $(\cdot,\cdot)_{D,K}$ is positive Hermitian
(sometimes called semi-inner product) on~$H^1(D)$ and an inner product on~$X(D)$.
\end{lem}
\begin{proof}
The symmetry of $(\cdot,\cdot)_{D,K}$ is obvious. Its bilinearity follows from
$\phi^D_{\lambda v+w}=\lambda\phi^D_v+\phi^D_w$ for all $\lambda\in\R$, $v,w\in H^1(D)$.
For the positive definiteness assume that~$(v,v)_{D,K}=0$ for $v\in
X(D)$. From~\eqref{eq:ED2} we obtain that $\phi^D_v=0$ in~$D$. Since
$\phi_v^D\in H^1(D)$, also $\phi_v^D|_{\partial D}=0$. Therefore, $v\in H^1(D)$ satisfies
$\op{L} v=0$ in~$D$ and $v=\phi_v^D=0$ on $\partial D$, which
yields $v=0$ in $D$.
\end{proof}

The semi-normed space $(H^1(D),\tnrm{\cdot}_{D,K})$ is a topological
space induced by the semi-norm.
\begin{lem}\label{lem:Xclosed}
The space $X(D)$ is closed in $(H^1(D), \tnrm{\cdot}_{D,K})$.
\end{lem}

\begin{proof}
Let $\{v_k\}_{k\in\N}\subset X(D)$ converge to $v\in H^1(D)$ with
respect to $\tnrm{\cdot}_{D,K}$. From~\eqref{eq:ED2} we obtain
\[\tnrm{v_k-v}_{D,D}\leq
c\tnrm{v_k-v}_{D,K}\to0.\]
In particular, we have that $\norm{\nabla(\phi_{v_k}^D-\phi_v^D)}_{L^2(D)}\to0$.
Hence, for $\fie\in H_0^1(D)$ it holds that
\[
|a(v,\fie)|=|a_\Delta(\phi_v^D,\fie)|\leq
\underbrace{|a_\Delta(\phi_{v_k}^D,\fie)|}_{=0}+|a_\Delta(\phi_{v_k}^D-\phi_v^D,\fie)|\leq
\norm{\nabla(\phi_{v_k}^D-\phi_v^D)}_{L^2(D)}\norm{\nabla \fie}_{L^2(D)}\to0,
\]
which shows that $a(v,\fie)=0$.
Finally, $v|_{\partial D\cap\partial\Omega}=\phi_v^D|_{\partial
  D\cap\partial\Omega}=0$ proves $v\in X(D)$.
\end{proof}

The semi-inner product $(\cdot,\cdot)_{D,K}$ on $H^1(D)$ is sufficient to define an element of
best approximation in the closed subspace~$X(D)$
(cf.~Lemma~\ref{lem:Xclosed}) of~$H^1(D)$. The positive definiteness of
$(\cdot,\cdot)_{D,K}$ on~$X(D)$ (cf.~Lemma~\ref{lem:skp}) implies that any element of best
approximation is unique. Hence, the $(\cdot,\cdot)_{D,K}$-orthogonal projection
\[\op{P}_{D,K}:H^1(D)\to X(D)\]
is well-defined. 

The following equivalence relation will be useful at the end of this
section, when error estimates with respect to $\tnrm{\cdot}_{D,D}$ are reformulated in the usual $L^2(D)$-norm.

\begin{lem} \label{lem:aeq}
There are constants $c_1,c_2>0$ (depending on $A$) such that 
\[
\frac{1}{c_1}\norm{v}_{1,D}\leq \tnrm{v}_{D,D}\leq c_2\norm{v}_{1,D}
\]
for all $v\in H^1(D)$. It holds that $c_1\sim \kappa_A$ and $c_2\sim\lambda_{\max}$.
\end{lem}
\begin{proof}
The choice $\fie=z_v^D$ in \eqref{eq:transfp} shows
\[
\norm{\nabla z_v^D}_{L^2(D)}^2=\int_D \nabla v^T(A-I)\nabla
z_v^D\ud x\leq |\lambda_{\max}-1|\norm{\nabla v}_{L^2(D)}\norm{\nabla z_v^D}_{L^2(D)}
\]
and hence $\norm{\nabla z_v^D}_{L^2(D)}\leq
|\lambda_{\max}-1|\norm{\nabla v}_{L^2(D)}$. With
$c_A:=|\lambda_{\max}-1|+1\leq \max\{2,\lambda_{\max}\}$ we obtain
\begin{equation}\label{eq:rema1}
\norm{\nabla\phi_v^D}_{L^2(D)}\leq \norm{\nabla
  z_v^D}_{L^2(D)}+\norm{\nabla v}_{L^2(D)}\leq
c_A\norm{\nabla v}_{L^2(D)}.
\end{equation}
In addition, Poincar\'e's inequality leads to
\[
\norm{z_v^D}_{L^2(D)}\leq c_P\,\diam\,D\,\norm{\nabla z_v^D}_{L^2(D)}\leq
c_P\,\diam\,D\,|\lambda_{\max}-1|\norm{\nabla v}_{L^2(D)},
\]
which shows
\begin{equation}\label{eq:rema2}
\norm{\phi_v^D}_{L^2(D)}\leq
\norm{v}_{L^2(D)}+c_Pc_A\,\diam\,D\norm{\nabla v}_{L^2(D)}.
\end{equation}
From \eqref{eq:rema1} and \eqref{eq:rema2} we obtain
\begin{align*}
\tnrm{v}_{D,D}^2&\leq
\left(\norm{v}_{L^2(D)}+c_Pc_A\,\diam\,D\norm{\nabla  v}_{L^2(D)}\right)^2+c_A^2
(\diam\,D)^2\norm{\nabla v}^2_{L^2(D)}\\
&\leq 2\norm{v}^2_{L^2(D)}+2c_P^2c_A^2 (\diam\,D)^2\norm{\nabla
    v}_{L^2(D)}^2+c_A^2 (\diam\,D)^2\norm{\nabla v}^2_{L^2(D)}
\leq c_2^2\norm{v}_{1,D}^2
\end{align*}
if we set $c_2:=\max\{\sqrt{2},c_A\sqrt{2c_P^2+1}\}$.
Similarly, testing \eqref{eq:transfp} with $\fie=z^D_v$ one has
$a(z_v^D,z_v^D)=a(\phi_v^D,z_v^D)-a_\Delta(\phi_v^D,z_v^D)$
and thus
\begin{align*}
\norm{v}_{1,D}^2&\leq
\left(\norm{\phi_v^D}_{L^2(D)}+\frac{c_P c_A}{\lambda_{\min}}\,\diam\,D\,\norm{\nabla
  \phi_v^D}_{L^2(D)}\right)^2
+\left(\frac{c_A}{\lambda_{\min}}\right)^2(\diam\,D)^2\norm{\nabla\phi_v^D}_{L^2(D)}^2\\
&\leq 2\norm{\phi_v^D}^2_{L^2(D)}+\left(\frac{c_A}{\lambda_{\min}}\right)^2(2c_P^2+1)(\diam\,D)^2\norm{\nabla
  \phi_v^D}^2_{L^2(D)}
\leq c_1^2\tnrm{v}_{D,D}^2
\end{align*}
with $c_1:=\max\{\sqrt{2},\frac{c_A}{\lambda_{\min}}\sqrt{2c_P^2+1}\}$.
\end{proof}

\subsection{Approximation from finite-dimensional spaces}
The approximation of a given element~$u\in X(D)$ with respect to~$\tnrm{\cdot}_{D,K}$ can be related to the
approximation of the harmonic function~$\phi_u^D$ with respect to~${\norm{\cdot}_{1,K}}$.
\begin{lem} \label{lem:transf}
Let $V_p=\spann\{\fie_1,\dots,\fie_p\}$,
$W_p=\spann\{\psi_1,\dots,\psi_p\}$ be $p$-dimensional subspaces of
$H^1(D)$ such that
\begin{equation}\label{eq:psifie}
\op{L} \fie_i=-\Delta \psi_i\quad \text{and} \quad \fie_i|_{\partial D}=\psi_i|_{\partial D},\quad i=1,\dots,p.
\end{equation}
Then 
\[
\inf_{v\in V_p}\tnrm{u-\op{P}_{D,K}v}_{D,K}\leq\inf_{w\in W_p}\norm{\phi_u^D-w}_{1,K}
\]
for all $K\subset D$.
\end{lem}
\begin{proof}
Let $w=\sum_{i=1}^p\gamma_i \psi_i\in W_p$ with $\gamma_i\in\R$. Then $v:=\sum_{i=1}^p\gamma_i\fie_i\in V_p$ satisfies
$\op{L} v=-\Delta w$ in $D$ and $v=w$ on~$\partial D$. Hence, $\phi^D_v=w$ and from $\phi^D_{u-v}=\phi^D_u-\phi^D_v$ we obtain
\[
\tnrm{u-\op{P}_{D,K}v}_{D,K}=\tnrm{\op{P}_{D,K}(u-v)}_{D,K}\leq \tnrm{u-v}_{D,K}
=\norm{\phi^D_u-\phi^D_v}_{1,K}=\norm{\phi^D_u-w}_{1,K},
\]
which proves that $\inf_{v\in V_p}\tnrm{u-\op{P}_{D,K}v}_{D,K}\leq
\norm{\phi_u^D-w}_{1,K}$. The assertion follows since $w\in W_p$ is arbitrary.
\end{proof}

The latter property will now be used to construct finite-dimensional
spaces~$X_k$ which approximate~$X(D)$ independently of~$A$.
For deriving error estimates it is required that
$\phi_u^D$ has a slightly higher regularity.
The following interior regularity result is proved similar to
\cite[Thm.~8.8]{GT}. While the latter result expects a positive
distance $\dist(K,\partial D)$, our modification assumes only $\sigma(K,D)>0$.
\begin{lem} \label{lem:H2phi}
For any subset~$K\subset D$ satisfying
$\sigma:=\sigma(K,D)>0$ it holds that $\phi_u^D\in H^2(K)$ and
\[
\norm{\partial_i\nabla \phi_u^D}_{L^2(K)}
\leq
\frac{2}{\sigma}\norm{\partial_i \phi_u^D}_{L^2(D)},\quad i=1,\dots,d.
\]
\end{lem}
\begin{proof}
Due to $u\in X(D)$, we have that $\phi_u^D\in H^1(D)$ is harmonic and $\phi_u^D=0$ on $\partial
D\cap\partial\Omega$.
Let $\xi\in C^\infty(D)$ satisfy $0\leq\xi\leq 1$, $\xi=1$ in $K$,
$\xi=0$ in a neighborhood of~$\partial D\cap\Omega$ and
\[\norm{\nabla\xi}_{\infty,D}\leq \frac{2}{\sigma}.\]
Since $\xi^2 \phi_u^D=0$ on $\partial D=(\partial D\cap\Omega)\cup
(\partial D\cap\partial\Omega)$, we have
$\hat K:=\supp\xi^2\phi_u^D\subset D$. For $0<h<\frac{1}{2}\dist(\hat K,\partial
D)$ define
$\fie:=\partial_i^{-h}(\xi^2\partial_i^h\phi_u^D)\in H_0^1(D)$,
where 
\[\partial_i^hu(x) :=\frac{u(x+he_i)-u(x)}{h}\]
denotes the difference quotient of $u$ in direction $i$.
Due to the dense imbedding of $C_0^\infty(D)$ in
$H_0^1(D)$, we may assume that $\fie\in C_0^\infty(D)$.
Since $\phi_u^D$ is harmonic, from
\begin{align*}
0&=-\int_D \nabla \phi_u^D\cdot\nabla[\partial_i^{-h}(\xi^2\partial_i^h\phi_u^D)]\ud x
=\int_D \nabla \partial_i^h\phi_u^D\cdot\nabla[\xi^2 \partial_i^h\phi_u^D]\ud x\\
&= \int_D \nabla \partial_i^h \phi_u^D\cdot [\xi\nabla (\xi\,\partial_i^h\phi_u^D)+\xi\nabla\xi\,\partial_i^h\phi_u^D]\ud x\\
& = \int_D [\nabla (\xi\partial_i^h \phi_u^D)-\nabla\xi\,\partial_i^h
\phi_u^D]\cdot [\nabla
(\xi\partial_i^h\phi_u^D)+\nabla\xi\, \partial_i^h\phi_u^D]\ud x\\
&=\int_D |\nabla
(\xi\partial_i^h \phi_u^D)|^2-|\nabla\xi \,\partial_i^h \phi_u^D|^2\ud x
\end{align*}
we obtain that
\[
\norm{\nabla\partial^h_i\phi_u^D}^2_{L^2(K)}\leq
\int_D |\nabla (\xi\partial_i^h \phi_u^D)|^2\ud x
=\int_D |\nabla\xi\,\partial_i^h\phi_u^D|^2\ud x
\leq \frac{4}{\sigma^2}\norm{\partial_i^h\phi_u^D}_{L^2(\hat K)}^2.
\]
Hence, we have shown the desired estimate for the finite
differences. The estimate for the differential operators follow from applying two results from~\cite{GT}.
With Lemma~7.23 from~\cite{GT} we obtain
$\norm{\partial_i^h\phi_u^D}_{L^2(\hat K)}\leq\norm{\partial_i\phi_u^D}_{L^2(D)}$.
By Lemma~7.24 from~\cite{GT} it follows that $\phi_u^D\in H^2(K)$ and
\[
\norm{\partial_i\nabla\phi_u^D}_{L^2(K)}\leq
\frac{2}{\sigma}\norm{\partial_i\phi_u^D}_{L^2(D)}.
\]
\end{proof}

Assume that $\Omega$ is polyhedral, and let $\triangle_H$ be a
quasi-uniform polyhedrization of~$\Omega$ with mesh size~$H>0$. We
define the space
\[
W_p:=\{v|_D,\,v\in \op{S}^{1,0}(\triangle_H)\}
\]
of piecewise linear finite elements
$\op{S}^{1,0}(\triangle_H)$ restricted to~$D$.
Obviously, the mesh size~$H$ is connected with the dimension~$p$ of~$W_p$ as 
\begin{equation}\label{eq:HD}
H\leq c_R\frac{\diam\,D}{\sqrt[d]{p}}
\end{equation}
with a constant~$c_R>0$.
Let $K_H\subset\triangle_H$ be the smallest polyhedrization such
that $K\subset K_H\subset D$ and $\sigma(K_H,D)>0$. Note that we are
interested in the limit $H\to0$.
Due to $\phi_u^D\in H^2(K_H)$ (see Lemma~\ref{lem:H2phi}), the nodal interpolation
operator~$\mathfrak{I}_H:L^2(D)\to W_p$ provides the following interpolation
error estimates
\begin{equation}\label{eq:cle}
\norm{\phi_u^D-\mathfrak{I}_H\phi_u^D}_{L^2(K_H)}\leq
c_{\mathfrak{I}} H|\phi_u^D|_{H^1(K_H)},\quad
\norm{\phi_u^D-\mathfrak{I}_H\phi_u^D}_{H^1(K_H)}\leq
c_{\mathfrak{I}} H|\phi_u^D|_{H^2(K_H)}.
\end{equation}

Let another set $K'\subset D$ satisfy $K\subset K_H\subset K'$ such that $\sigma(K,K')>0$.
We set
\[
\rho:=\frac{\diam\,D}{\sigma(K,K')}.
\]
Using $\sigma(K_H,K')\geq\sigma(K,K')-H$ and $H\leq
c_Rp^{-1/d}\,\diam\,D=c_Rp^{-1/d}\rho\,\sigma(K,K')$, we obtain
\[
\sigma:=\sigma(K_H,K')\geq
\left(1-\frac{c_R\rho}{\sqrt[d]{p}}\right)\sigma(K,K')\geq \frac{1}{2}\sigma(K,K')
\]
for $p\geq(2c_R\rho)^d$.
From \eqref{eq:cle}, Lemma~\ref{lem:ipoi} applied to $A=I$, and Lemma~\ref{lem:H2phi} we obtain 
\begin{align*}
\inf_{w\in W_p}\norm{\phi^D_u-w}_{1,K_H}^2
&\leq c^2_{\mathfrak{I}} H^2\left(|\phi^D_u|^2_{H^1(K_H)}+(\diam\,K_H)^2|\phi^D_u|^2_{H^2(K_H)}\right)\\
&\leq 4\frac{c^2_{\mathfrak{I}}  H^2}{\sigma^2}\left(\norm{\phi^D_u}^2_{L^2(K')}
+d(\diam\,K_H)^2\norm{\nabla \phi_u^D}_{L^2(K')}^2\right)
\end{align*}
and hence
\begin{equation}\label{eq:HD1}
\inf_{w\in W_p}\norm{\phi^D_u-w}_{1,K}\leq\inf_{w\in W_p}\norm{\phi^D_u-w}_{1,K_H}
\leq \frac{4dc_{\mathfrak{I}} H}{\sigma(K,K')}\norm{\phi_u^D}_{1,K'}.
\end{equation}

If $\psi_1,\dots,\psi_p$ denotes a basis of $W_p$ and $\fie_1,\dots,\fie_p$
are the corresponding solutions of~\eqref{eq:psifie}, then
Lemma~\ref{lem:transf}, \eqref{eq:HD}, and \eqref{eq:HD1} show the
existence of an at most $p$-dimensional space
\[Y_K:=\op{P}_{D,K}V_p\subset X(D)\] 
with $p\geq (2c_R\rho)^d$ such that
\begin{equation}\label{eq:Xk}
\inf_{v\in Y_K}\tnrm{u-v}_{D,K}\leq \frac{4dc_{\mathfrak{I}} H}{\sigma(K,K')}\norm{\phi_u^D}_{1,K'} 
\leq \frac{4dc_{\mathfrak{I}}c_R}{\sqrt[d]{p}}\rho\,\norm{\phi_u^D}_{1,K'}
=\frac{c_S}{\sqrt[d]{p}}\rho\,\tnrm{u}_{D,K'}
\end{equation}
with the $A$-independent constant $c_S:=4dc_{\mathfrak{I}}c_R$.

For the following theorem the algebraic decay \eqref{eq:Xk} is exploited
recursively on a sequence of nested domains to obtain the desired exponential convergence.

\begin{thm} \label{thm:exappr}
Let $K\subset D$ such that $\eta\,\sigma(K,D)\geq\diam\,D$ with some $\eta>0$. For any $\eps>0$ there is a $k$-dimensional subspace $X_k\subset X(D)$ satisfying
\begin{equation}\label{(dist u}
\inf_{v\in X_k} \tnrm{u-v}_{D,K}\leq \eps\tnrm{u}_{D,D}
\quad\text{for all }u\in X(D)
\end{equation}
provided that $k\geq c_\eta \lceil |\log \eps|\rceil^{d+1}$, where
$c_\eta:=\lceil\eta\max\{c_S\e,2c_R\}\rceil^d$.
\end{thm}

\begin{proof}
Let $\ell=\lceil |\log \eps|\rceil$ and $r_0:=\sigma(K,D)$.
We consider a sequence of nested domains
$$K_j=\{x\in\Omega:\dist(x,K)\leq r_0(\ell-j)/\ell\}, \quad
j=0,\dots,\ell.$$
Notice that $K=K_\ell\subset K_{\ell-1}\subset\cdots\subset K_0\subset
D$ with
\[
\sigma(K_j,K_{j-1})=\frac{r_0}{\ell}=\frac{\sigma(K,D)}{\ell}\geq
\frac{\diam\,D}{\eta\ell}, \quad j=1,\dots,\ell.\]
According to \eqref{eq:Xk} (with the choice $K:=K_j$ and
$K':=K_{j-1}$), there is $Y_{K_j}\subset X(D)$, $\dim{Y_{K_j}}\leq p$,
so that for all $u\in X(D)$
\begin{align} \label{eq:ap}
\begin{split}
\inf_{v\in Y_{K_j}} \tnrm{u-v}_{D,K_j}
&\leq \frac{c_S}{\sqrt[d]{p}}\eta\ell\,\tnrm{u}_{D,K_{j-1}}
\leq \eps^{1/\ell}\tnrm{u}_{D,K_{j-1}},
\end{split}
\end{align}
if we choose $p\geq p_0:=\lceil \eta\ell\max\{c_S \eps^{-1/\ell},2c_R\}\rceil^d$.

Let $e_0:=u\in X(D)$. Estimate \eqref{eq:ap} defines an
element~$v_1\in Y_{K_1}$ and hence $e_1:=e_0-v_1\in X(D)$ so that
\[\tnrm{e_1}_{D,K_1}\leq
\eps^{1/\ell}\,\tnrm{e_0}_{D,K_0}.\]
Similarly, from \eqref{eq:ap}  we obtain
approximants $v_j\in Y_{K_j}$, $j=2,\dots,\ell$, so that with $e_j:=e_{j-1}-v_j$
\[
\tnrm{e_j}_{D,K_j}\leq \eps^{1/\ell}\,\tnrm{e_{j-1}}_{D,K_{j-1}}.
\]
Since $e_0=e_\ell+\sum_{j=1}^\ell v_j$ and $\sum_{j=1}^\ell v_j\in
X_k:=\bigoplus_{j=1}^\ell Y_{K_j}$, we are led to
\[
\inf_{v\in X_k}\tnrm{e_0-v}_{D,K_\ell}\leq \tnrm{e_\ell}_{D,K_\ell}
\leq(\eps^{1/\ell})^\ell \tnrm{e_0}_{D,K_0}
\leq \eps \tnrm{u}_{D,D}.
\]
The dimension $k$ of $X_k$ is bounded by $p\ell$.
From $\eps^{-1/\ell}\leq \e$ we obtain that
\[
p_0\leq \lceil \eta\ell\max\{c_S\e,2c_R\}\rceil^d,
\]
which proves the assertion. 
\end{proof}

\section{Separable approximation of solution operators}
\label{sec:innreg}
In \cite{MBWHInv,InvFEM2} we were able to prove that inverse FE stiffness matrices of scalar elliptic boundary value
problems can be approximated using hierarchical matrices with
logarithmic-linear complexity.
Since we do not want to repeat the proofs from \cite{MBWHInv}, we
concentrate on the central problem of constructing degenerate kernel
expansions
\[G_k(x,y):=\sum_{i=1}^k u_i(x) v_i(y)\]
to Green's function~$G$ for the operator $\op{L}$ satisfying
\begin{itemize}
\item[(i)] $\op{L} G(x,\cdot)=\delta_x$ in $\Omega$,
\item[(ii)] $\op{L} G(x,\cdot)=0$ on $\partial\Omega$
\end{itemize}
for all $x\in\Omega$; for the existence of $G$ see~\cite{MR83h:35033}.
The rest of the proof in \cite{MBWHInv} is based on the following
existence result (see Theorem~3.5 in \cite{InvFEM2}) and
can be applied without changes. 

\begin{thm} \label{thm:old}
Let $D_1,D_2\subset\Omega$ be two domains such that $D_2$ is convex and
\begin{equation}\label{eq:adm1}
\eta\,\dist(D_1,D_2)\geq\diam\,D_2.
\end{equation}
Then for any $\eps>0$ there is a separable approximation
\[
G_k(x,y)=\sum_{i=1}^ku_i(x)v_i(y)\quad\text{with }k\leq c_{\eta,A}|\log\eps|^{d+1}
\]
satisfying 
\[
\norm{G(x,\cdot)-G_k(x,\cdot)}_{L^2(D_2)}\leq\eps\norm{G(x,\cdot)}_{L^2(\hat{D}_2)}
\quad \text{for all }x\in D_1,
\]
where $\hat{D}_2:=\{y\in\Omega:2\eta\,\dist(y,D_2)\leq\diam\,D_2\}$.
\end{thm}

It is remarkable that this result holds
for arbitrary coefficients $a_{ij}\in L^\infty(\Omega)$
of the operator~$\op{L}$ satisfying $\lambda_{\min}>0$ because in this case Green's function does not
possess any higher regularity. 
Although the approximation is independent of the smoothness of the
coefficient~$A$, it can be seen from the proof in \cite{InvFEM2} that the estimate on~$k$ 
(i.e.\ the constant $c_{\eta,A}$) depends on the contrast~$\kappa_A$
as~$\kappa_A^{d/2}$.

In this section it will be proved using
the $\tnrm{\cdot}_{D,K}$-norm that the dependence on $\kappa_A$ can be avoided. Note that the approximation technique
presented in this article applies not only to operators \eqref{eq:op1}.
In \cite{InvFEM2} we considered general second order elliptic scalar
operators and in~\cite{bebeMaxwell} an analogous result was proved for the
curl-curl operator. For practical reasons it is important to remark
that also the Schur complement and the factors of the
LU decomposition can be handled by hierarchical matrices with
logarithmic-linear complexity. This paves the way to approximate
direct methods that do not suffer from the well-known effect of fill-in.
The proof \cite{MBLU} is based on the
approximation of the inverse. Hence, also the results in \cite{MBLU} directly
benefit from the new estimates of this article.

The next theorem is a variant of Theorem~\ref{thm:old}, in which the
approximation error (or the rank of the approximation) does not depend
on the coefficient~$A$. To this end, the results obtained in the previous sections for arbitrary elements in $X(D)$ will now be applied to a particular element.
Although the entries of~$A$ are only measurable, it can be proved that
$g_x:=G(x,\cdot)\in H^1(\Omega\ohne B_\eps(x))$ for all $\eps>0$;
see~\cite{MR83h:35033}. Hence, $g_x|_D$ is in~$X(D)$ for all $x\in \Omega\setminus \overline{D}$.

\begin{thm} \label{thm:indA}
Let $D_1,D_2\subset\Omega$ satisfy \eqref{eq:adm1}.
Then for any $\eps>0$ there is a separable approximation
\[
G_k(x,y)=\sum_{i=1}^k u_i(x) v_i(y)\quad\text{with }
k\leq c_{\eta}\lceil|\log \eps|\rceil^{d+1},
\]
so that for all $x\in D_1$
\begin{equation} \label{(supG-Gk}
\tnrm{G(x,\cdot)-G_k(x,\cdot)}_{\hat D_2,D_2}
\leq\eps\tnrm{G(x,\cdot)}_{\hat{D}_2,\hat{D}_2},
\end{equation}
where $\hat{D}_2$ is defined in Theorem~\ref{thm:old} and
$c_\eta := \lceil 2(\eta+1)\max\{c_S \e,2c_R\}\rceil^d$.
\end{thm}

\begin{proof}
Notice that because of
$\dist(D_1,\hat{D}_2)>0$, we have $g_x|_{\hat{D}_2}\in X(\hat{D}_2)$
for all $x\in D_1$. Since $\diam\,\hat{D}_2\leq (1+1/\eta)\,\diam\,D_2$, it holds that
\[
\sigma(D_2,\hat{D}_2)=\frac{1}{2\eta}\diam\,D_2\geq\frac{1}{2(\eta+1)}\diam\,\hat{D}_2.
\]
Hence, Theorem~\ref{thm:exappr} can be applied with
$K:=D_2$, $D:=\hat D_2$, and $\eta'=2(\eta+1)$.

Let $\{v_1,\dots,v_k\}$ be a basis of the
subspace $X_k\subset X(\hat{D}_2)$ with $k\geq c_{\eta'}
\lceil|\log \eps|\rceil^{d+1}$.
By means of~\eqref{(dist u} we can decompose $g_x$ as
$g_x=\tilde{g}_x+e_x$ with $\tilde{g}_x\in X_k$ and
\[\tnrm{e_x}_{\hat{D}_2,D_2}\leq\eps\tnrm{g_x}_{\hat{D}_2,\hat{D}_2}.\]
Expressing $\tilde{g}_x$ in the basis of $X_k$, we obtain
\[
\tilde{g}_x=\sum_{i=1}^k u_i(x) v_i
\]
with coefficients $u_i(x)$ depending on $x\in D_1$.
The function $G_k(x,y):=\tilde{g}_x(y)$ satisfies \eqref{(supG-Gk}.
\end{proof}

The previous theorem shows that $k$ is bounded independently of the coefficient~$A$. Estimate~\eqref{(supG-Gk}, however, is an estimate
with respect to the $\tnrm{\cdot}_{D,K}$-norm.
If the arguments from \cite{MBWHInv} are to be applied in order to estimate
the rank of $\H$-matrix approximations, the usual $L^2$-norm (as in
Theorem~\ref{thm:old}) has to be
used.

\begin{thm}
Let $D_1,D_2\subset\Omega$ satisfy \eqref{eq:adm1}.
Then for any $\eps>0$ there is a separable approximation
\[
G_k(x,y)=\sum_{i=1}^k u_i(x) v_i(y)\quad\text{with}\quad
k\leq c_{\eta}\lceil|\log (\eps \kappa_A^{-3/2}/\lambda_{\max})|\rceil^{d+1},
\]
so that for all $x\in D_1$
\[
\norm{G(x,\cdot)-G_k(x,\cdot)}_{L^2(D_2)}
\leq \eps\norm{G(x,\cdot)}_{L^2(\hat{D}_2)},
\]
where $\hat{D}_2$ is defined in Theorem~\ref{thm:old} and
$c_\eta := \lceil (4\eta+2)\max\{c_S \e,2c_R\}\rceil^d$.
\end{thm}

\begin{proof}
As in the proof of Theorem~\ref{thm:indA} we have $g_x|_{\hat{D}_2}\in X(\hat{D}_2)$ for all
$x\in D_1$. Let
\[D'_2:=\{y\in\Omega:4\eta\,\dist(y,D_2)\leq\diam\,D_2\}.\]
Then $D_2\subset D_2'\subset \hat{D}_2$, $\diam\,D'_2\leq (1+1/(2\eta))\,\diam\,D_2$, and 
\[
\sigma(D_2,D'_2)=\frac{1}{4\eta}\diam\,D_2\geq \frac{1}{4\eta+2}\diam\,D'_2.
\]
Hence, Theorem~\ref{thm:exappr} can be applied with
$K:=D_2$, $D:=D'_2$, and $\eta'=4\eta+2$, which yields the
subspace $X_k\subset X(D'_2)$ with $k\geq c_{\eta'}
\lceil|\log \eps|\rceil^{d+1}$.
By means of~\eqref{(dist u} we can decompose $g_x$ as
$g_x=\tilde{g}_x+e_x$ with $\tilde{g}_x\in X_k$ and $e_x\in X(D'_2)$
such that
\begin{align*}
\norm{e_x}_{L^2(D_2)}&\leq
c_1\tnrm{e_x}_{D'_2,D'_2}\leq  c_1 c_{D'_2,D_2}\tnrm{e_x}_{D'_2,D_2}\\
&\leq c_1 c_{D'_2,D_2}\eps\tnrm{g_x}_{D'_2,D'_2}
\leq c_1 c_2 c_{D'_2,D_2} \eps\norm{g_x}_{1,D'_2}\\
&=c_1 c_2 c_{D'_2,D_2}\eps\sqrt{\norm{g_x}^2_{L^2(D'_2)}+(\diam\,D'_2)^2
  \norm{\nabla g_x}^2_{L^2(D'_2)}},
\end{align*}
where we used Lemma~\ref{lem:aeq} and \eqref{eq:ED2}. From \eqref{eq:uc} we obtain that
\[
\norm{\nabla g_x}_{L^2(D'_2)}\leq \frac{2\sqrt{\kappa_A}}{\sigma(D'_2,\hat{D}_2)}\norm{g_x}_{L^2(\hat{D}_2)}.
\]
The dependence of $c_{D,K}$ on the domains $D$ and $K$ is not
explicitly known. A scaling argument reveals that \[c_{D,K}\leq
\hat c_{D,K}\left(\frac{\diam\,D}{\diam\,K}\right)^{d/2},\] where
$\hat c_{D,K}>0$ depends on the
shapes of $D$ and $K$ but not on their diameters. 
Due to $\sigma(D'_2,\hat{D}_2)=\sigma(D_2,D'_2)$, we have
\[
\frac{\diam\,D'_2}{\sigma(D'_2,\hat{D}_2)}=4\eta+2,
\]
\[
c_{D'_2,D_2}\leq
\hat c_{D'_2,D_2}\left(\frac{\diam\,D'_2}{\diam\,D_2}\right)^{d/2}
\leq \hat c_{D'_2,D_2}\left(1+\frac{1}{2\eta}\right)^{d/2}
\]
and we obtain from $c_1\sim\kappa_A$ and $c_2\sim\lambda_{\max}$ that
\begin{align*}
\norm{e_x}_{L^2(D_2)}&\leq c_1 c_2 \hat c_{D'_2,D_2}\left(1+\frac{1}{2\eta}\right)^{d/2} \eps
\sqrt{1+4\kappa_A(4\eta+2)^2}\,\norm{g_x}_{L^2(\hat{D}_2)}\\
&\leq c \lambda_{\max}\kappa_A^{3/2} \eps \norm{g_x}_{L^2(\hat{D}_2)}.
\end{align*}
The assertion follows with the same arguments as used in the proof of Theorem~\ref{thm:indA}.
\end{proof}

Compared with the ``old'' Theorem~\ref{thm:old}, in which $k$ depends
on $\kappa_A$ as $k\sim \kappa_A^{d/2}|\log\eps|^{d+1}$, the new proof
via $A$-dependent norms $\tnrm{\cdot}_{D,K}$ significantly  improves
the rank estimate in the $L^2$-norm to
\[k\sim
|\log(\eps\kappa_A^{-3/2}/\lambda_{\max})|^{d+1},\] 
i.e., the contrast $\kappa_A$ enters the complexity of $\H$-matrices only
logarithmically.

\section{Numerical experiments} \label{sec:num}
In this section the influence of the coefficient~$A$ in the
differential operator \eqref{eq:op1} 
will be investigated numerically on the unit square $\Omega:=[0,1]^3$ in $\R^3$.
To this end, we choose $r$ balls $\Omega_i=B_{s_i}(x_i)$, $i=1,\dots,r$,
centered at randomly generated centers $x_i\in\Omega$; see Fig.~\ref{fig:cubsp}.
\begin{figure}[htb]
\centering
\includegraphics[angle=0,height=7cm]{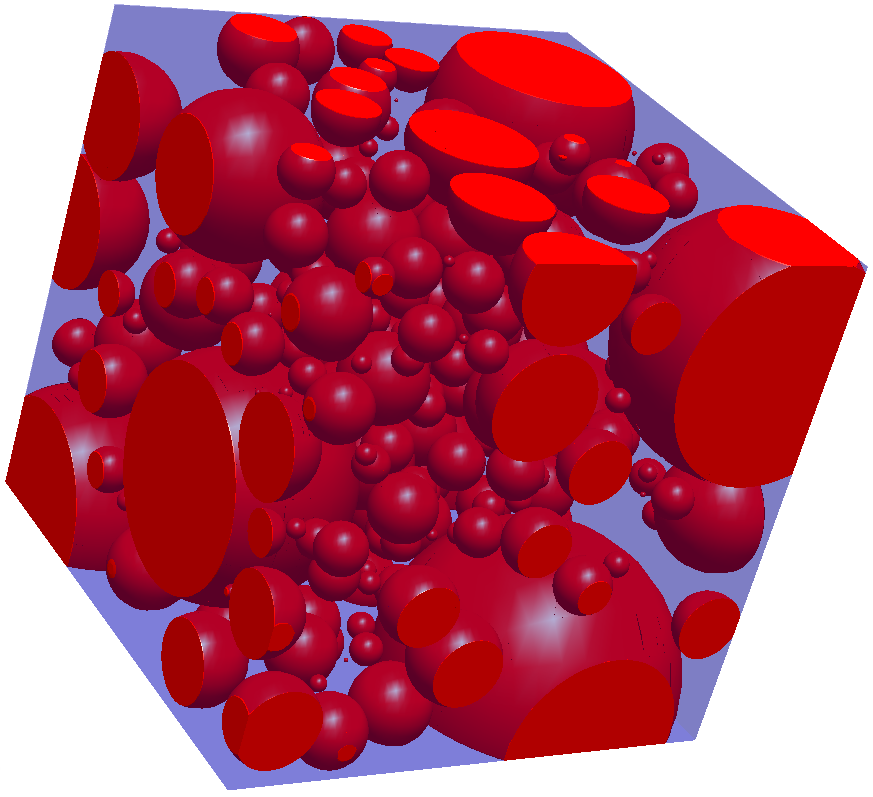}
\caption{$p=200$ inclusions with random coefficients in $[0,M]$.} \label{fig:cubsp}
\end{figure}
For the coefficient $A=\alpha I$ we use
\[
\alpha(x)=\begin{cases}
\alpha_i, & x\in \Omega_i,\\
1, & \text{else,}
\end{cases}
\]
where each $\alpha_i\in[0,M]$ is randomly chosen.
Table \ref{tab:hc3} contains the storage requirement per degree
of freedom when approximating the
inverse of standard FE discretizations using hierarchical matrices
with relative accuracy $\eps=10^{-3}$. The emphasis of these tests is not on
the scaling behavior, i.e.\ the dependence of the complexity with respect to
the number of degrees of freedom~$n$. Such tests were published in
\cite{MBWHInv,InvFEM2,bebeMaxwell} and in many other articles.
Here, we are primarily interested in the
dependence of the rank $k$ on the number of domains~$r$ and on the amplitude $M$
of the coefficient $\alpha$. Therefore, only one choice of $n$ will be
considered. The table shows the storage requirement of the
hierarchical matrix approximation to the inverse. Since the storage
depends linearly on $k$, this allows a direct comparison.
\begin{table}[htb]
\centering
\begin{tabular}{r|rrrrr}
$p\setminus M$ & 1 & 10 & 100 & 1000 & 10000 \\
\hline
1 & 22.92 & 22.84 & 22.63 & 22.54 & 22.07 \\
10 & 22.54 & 22.69 & 22.42 & 18.94 & 17.34 \\
100 & 22.74 & 22.84 & 23.58 & 18.15 & 15.51\\
1000 & 22.83 & 22.59 & 22.12 & 16.00 & 10.78
\end{tabular}
\caption{Storage per degree of freedom in kByte for $n=166\,375$.}
\label{tab:hc3}
\end{table}
Apparently, the storage requirement is bounded with respect to the number of
domains~$r$ and with respect to the amplitude~$M$. The complexity actually
improves with $p$ and $M$.

\section*{Appendix}
Let $D\subset\R^d$ be a bounded domain and let $K\subset D$ be a subdomain. We denote the dual space of~$H^1(D)$ by $\widetilde{H}^{-1}(D)$.

\begin{lem} \label{lem:one}
The Newton potential operator $N_K:\widetilde{H}^{-1}(D)\to H^1(D)$ defined by
\[
(N_K\fie)(y):=\int_K S(x,y)\fie(x)\ud x,\quad y\in D, 
\]
with 
\[S(x,y):=\frac{1}{4\pi}\frac{1}{|x-y|}\]
is continuous, i.e.\ $\norm{N_K\fie}_{H^1(D)}\leq
c_D\norm{\fie}_{\widetilde{H}^{-1}(D)}$.
\end{lem}
\begin{proof}
Since $\widetilde{H}^{-1}(D)$ can be
regarded as the closure of $C_0^\infty(D)$ with respect to 
$\norm{\cdot}_{H^{-1}(\R^d)}$,
we may consider $\fie\in C_0^\infty(D)$ and define
\[
u(x):=(N_K\fie)(x)=\int_K S(x,y)\fie(y)\ud y.
\]
Let $\xi\in
C_0^\infty([0,\infty))$ be a non-negative cut-off function satisfying
$\xi(t)=1$ for $t\in[0,2R]$ and $\xi(t)=0$ for $t>3R$, where $R>0$ is chosen such that $D\subset
B_R(0)$. With
\[
u_\xi(x):=\int_{\R^d} \xi(|x-y|)S(x,y)\fie(y)\ud y.
\]
it holds that $u_\xi(x)=u(x)$ for $x\in D$ and thus
$\norm{u}_{H^1(D)}=\norm{u_\xi}_{H^1(D)}\leq
\norm{u_\xi}_{H^1(\R^d)}$.
In \cite[p.\ 109]{STEINB_BOOK} it is proved that
\[
\norm{u_\xi}_{H^1(\R^d)}\leq c\norm{\fie}_{H^{-1}(\R^d)}
\]
with $c=c(R)$. The assertion follows
from
\[
\norm{\fie}_{H^{-1}(\R^d)}=\sup_{0\neq v\in H^1(\R^d)}
\frac{(\fie,v)_{L^2(\R^d)}}{\norm{v}_{H^1(\R^d)}}
\leq \sup_{0\neq v\in H^1(D)} \frac{(\fie,v)_{L^2(D)}}{\norm{v}_{H^1(D)}}=\norm{\fie}_{\widetilde{H}^{-1}(D)}.
\]
\end{proof}

\begin{lem} \label{lem:nrmKD}
Let $u\in H^1(D)$
be harmonic. Then the norm of $u$ in $D$ is bounded by the norm in the
$C^{1,1}$ domain $K\subset D$, i.e.
\[
\norm{u}_{1,D}\leq c_{D,K}\norm{u}_{1,K}
\]
with $c_{D,K}>0$ independent of $u$.
\end{lem}

\begin{proof}
Let $K^c:=D\setminus\overline{K}$.
Since $u$ is harmonic in $D$, it can be represented by its Cauchy
data
\[
u(x)= \begin{cases} (V\partial_\nu u)(x)-(Ku)(x), & x\in K,\\
-(V\partial_\nu u)(x)+(Ku)(x),& x\in K^c,\\
\end{cases}
\]
with the boundary integral operators
\[
(Vu)(x):=\int_{\partial K} S(x,y) u(y)\ud s_y\quad \text{and}\quad
(Ku)(x):=\int_{\partial K} \partial_{\nu,y} S(x,y) u(y)\ud s_y.
\]
Since $C_0^\infty(D)$ is dense in $\widetilde{H}^{-1}(D)$,
we may consider $\fie\in C_0^\infty(D)$ and estimate
\begin{align*}
\int_D u(x)\fie(x)\ud x &= 
\int_K u(x)\fie(x)\ud x+ \int_{K^c} u(x)\fie(x)\ud x\\
&=\int_{\partial K} (\partial_\nu u)(y)(N_K\fie)(y)\ud s_y
-\int_{\partial K} u(y) (\partial_\nu N_K\fie)(y)
\ud s_y\\
&\quad -\int_{\partial K} (\partial_\nu u)(y) (N_{K^c}\fie)(y)\ud s_y
+\int_{\partial K} u(y) (\partial_\nu N_{K^c}\fie)(y) \ud s_y\\
&\leq \norm{\partial_\nu u}_{H^{-1/2}(\partial
  K)}\norm{N_K\fie}_{H^{1/2}(\partial K)} + \norm{u}_{H^{1/2}(\partial
  K)}\norm{\partial_\nu N_K\fie}_{H^{-1/2}(\partial K)}\\
&\quad +\norm{\partial_\nu u}_{H^{-1/2}(\partial
  K)}\norm{N_{K^c}\fie}_{H^{1/2}(\partial K)} + \norm{u}_{H^{1/2}(\partial
  K)}\norm{\partial_\nu N_{K^c}\fie}_{H^{-1/2}(\partial K)}\\
&\leq 2c_K'\norm{u}_{H^1(K)}\left(\norm{N_K\fie}_{H^1(K)}+\norm{N_{K^c}\fie}_{H^1(K)}\right)\\
&\leq 2c_K'\norm{u}_{H^1(K)}\left(\norm{N_K\fie}_{H^1(D)}+\norm{N_{K^c}\fie}_{H^1(D)}\right)
\end{align*}
due to $\norm{u}_{H^{1/2}(\partial K)}\leq
\norm{u}_{H^1(K)}$ and
$\norm{\partial_\nu u}_{H^{-1/2}(\partial K)}\leq
c'_K\norm{u}_{H^1(K)}$. Lemma~\ref{lem:one} leads to
\[
\int_D u(x)\fie(x)\ud x\leq 4c'_Kc_D\norm{u}_{H^1(K)}\norm{\fie}_{\widetilde{H}^{-1}(D)}
\]
and thus
\[
\norm{u}_{H^1(D)}=\sup_{\fie\in
  \widetilde{H}^{-1}(D)}\frac{(u,\fie)_{L^2(D)}}{\norm{\fie}_{\widetilde{H}^{-1}(D)}}\leq 4c_K'c_D\norm{u}_{H^1(K)},
\]
which leads to the desired estimate $\norm{u}_{1,D}\leq
c_{D,K}\norm{u}_{1,K}$ with a constant $c_{D,K}>0$ depending on~$D$
and~$K$.
\end{proof}

\bibliographystyle{plain}

\end{document}